\documentclass[final,hidelinks,onefignum,onetabnum]{siamart220329}

\usepackage{lipsum}
\usepackage{amsfonts}
\usepackage{amsmath,amssymb}
\usepackage{graphicx}
\usepackage{epstopdf}
\usepackage{algorithmic}
\ifpdf
  \DeclareGraphicsExtensions{.eps,.pdf,.png,.jpg}
\else
  \DeclareGraphicsExtensions{.eps}
\fi

\newsiamremark{remark}{Remark}
\newsiamremark{problem}{Problem}
\newsiamremark{example}{Example}
\newsiamremark{assumption}{Assumption}
\newsiamremark{hypothesis}{Hypothesis}
\crefname{hypothesis}{Hypothesis}{Hypotheses}
\newsiamthm{claim}{Claim}

\headers{SIPG Method for Optimal Control of the Obstacle Problem}{H. Antil, R. Khandelwal, and U. Rakhimov}

\title{A Discontinuous Galerkin Method for Optimal Control of the Obstacle Problem\thanks{Submitted to the editors DATE.
\funding{This work is partially supported by NSF grant DMS-2110263 and the Air Force Office of Scientific
Research (AFOSR) under Award NO: FA9550-22-1-0248}}}

\author{Harbir Antil 
	\and Rohit Khandelwal \and Umarkhon Rakhimov
	\thanks{Department of Mathematical Sciences and the Center for Mathematics and Artificial Intelligence (CMAI), George Mason University, Fairfax, VA 22030, USA.	 (\email{hantil@gmu.edu}, \email{rkhandel@gmu.edu}, \email{urakhimo@gmu.edu}).}}

\usepackage{amsopn}

\usepackage{enumerate}

\ifpdf
\hypersetup{
  pdftitle={An Example Article},
  pdfauthor={D. Doe, P. T. Frank, and J. E. Smith}
}
\fi
\usepackage[textsize=small]{todonotes}
\newcommand{\vertiii}[1]{{\left\vert\kern-0.25ex\left\vert\kern-0.25ex\left\vert #1 
		\right\vert\kern-0.25ex\right\vert\kern-0.25ex\right\vert}}

\allowdisplaybreaks

\def\xA{\mathcal{A}}
\def\cP{\mathcal{P}}
\def\cW{\mathcal{W}}
\def\cQ{\mathcal{Q}}
\def\cL{\mathcal{L}}
\def\cK{\mathcal{K}}
\def\cA{\mathcal{A}}
\def\cN{\mathcal{N}}
\def\cH{\mathcal{H}}

\def\cV{\mathcal{V}}
\def\cM{\mathcal{M}}
\def\cT{\mathcal{T}}
\def\cE{\mathcal{E}}

\def\smean#1{\{\hskip -3pt\{#1\}\hskip -3pt\}}
\def\sjump#1{[\hskip -1.5pt[#1]\hskip -1.5pt]}

\begin{document}

\maketitle

\begin{abstract}
This article provides quasi-optimal a priori error estimates for an optimal control problem constrained by an 
elliptic obstacle problem where the finite element discretization is carried out using the symmetric interior 
penalty discontinuous Galerkin method. The main proofs are based on the improved $L^2$-error estimates 
for the obstacle problem, the discrete maximum principle, and a well-known quadratic growth property. The standard (restrictive) assumptions on mesh are not assumed here.
\end{abstract}



\section{Introduction}\label{sec:Intro}

Free boundary problems governed by physical principles are ubiquitous in science and 
engineering. In particular variational inequalities of first kind, such as obstacle problem, arise 
in elasticity, fluid filtration in porous media and finance. Special attention has been given to 
study the obstacle problem which acts as a model problem in many of these applications. 
It is also natural to consider optimization problems governed by variational inequalities as 
constraints which is the focus here. 

This article considers an optimal control problem governed by an elliptic variational
inequality (EVI) of the first kind (obstacle problem) \cite{glowinski1984numerical,rodrigues1987obstacle} as constraints.
Let $\Omega $ be a bounded domain in $\mathbb{R}^d$ where $d \in \{2,3\}$ with Lipschitz boundary 
$\partial \Omega$. For $(u,z)$ denoting the state-control pair, the optimal control problem is given by
\begin{equation} \label{problem1}
\tag{$A$} \hspace*{-2cm}
\begin{cases}
\text{min } \left\{J(u,z):= \frac{1}{2} \|u-u_d\|^2_{L^2(\widehat{\Omega})} + \frac{\nu}{2} \|z\|^2_{L^2(\Omega)} \right\}, & \\ 
\text{subject to } ~~a(u, v-u) \geq (z, v-u) \quad \forall~ v \in \cK, & \\ \text{and } u \in \cK: =\{w \in \cV:=H^1_0(\Omega)~:~w(x) \geq g(x)~\text{a.e. in } \Omega\},
\end{cases}
\end{equation}
\noindent
where the open set $\widehat{\Omega}$ is either ${\widehat{\Omega}= \Omega}$ or ${\widehat{\Omega}= \Omega_0 \subset \subset \Omega.}$ Here, $a(s,t):= \int_{\Omega} \nabla s \cdot \nabla t~dx~\forall s , t \in \cV$ is the continuous, $\cV$-elliptic bilinear form on $\cV \times \cV$, $(\cdot, \cdot)$ denotes the $L^2(\Omega)$ inner-product and $\|\cdot\|_{L^2(\widehat{\Omega})} $ is the norm on ${L^2(\widehat{\Omega})}$. Let $u_d \in {L^2(\widehat{\Omega})}$ be the desired state and $\nu >0 $ is the regularization parameter. The exact functional-analytic setting of Problem \eqref{problem1} is stated in the upcoming sections.

This paper derives a priori error estimates for Problem \eqref{problem1} using the symmetric interior penalty discontinuous Galerkin (DG) finite element method (FEM). The state $u$ represents either global matching state, i.e., $\widehat{\Omega}=\Omega$ or local matching state $\widehat{\Omega}=\Omega_0$ where $\Omega_0 \subset \subset \Omega$. In both cases, finite element error estimates for $z$ and $u$ are derived. In the local matching setting, error estimates are quasi-optimal i.e., $O(h)$ for control. In the global setting, we are limited by the DG boundary conditions and we obtain $O(h^{\frac12})$.

DG methods provide greater flexibility to handle complex geometries and allow discretizations with hanging nodes and different degrees of polynomial approximation on different elements. In addition, they are more flexible to locally adapt the discretization or the degree of the discrete basis functions which captures better approximation of the solution. For a general reference on DG methods, we refer to \cite{cockburn2012discontinuous}. Furthermore, the articles \cite{arnold1982interior,brezzi2000discontinuous,arnold2000discontinuous} discuss DG methods for elliptic problems. For related work on obstacle problem, we refer to \cite{WHC:2010:DGOP}. 

The literature on the optimal control problem constrained by PDEs is indeed significant (see \cite{falk1974error,FTroeltzsch_2010a,MHinze_RPinnau_MUlbrich_SUlbrich_2009a,KIto_KKunisch_1990b,HA:2018:PDEOCP}) but this is still a highly active research area. For works on a priori error analysis for optimal control problems governed by PDEs, we refer to \cite{falk1973approximation,geveci1979approximation,arada2002error,deckelnick2007convergence,leykekhman2012local}. Optimal control for the obstacle problem is a challenging problem due to non-differentiability of the solution mapping $S$ between the control and state variables. In the seminal works \cite{Mignot:1976:CEVI, Mignot:1984:OCPVI}, the authors demonstrate the directional differentiability of the map $S$ and, later introduced the optimality conditions of the strong stationarity type for Problem \eqref{problem1}. The authors in \cite{Kunisch:2012:SOPS} built upon the ideas from \cite{Mignot:1976:CEVI} and derived the second order strong stationarity conditions for Problem \eqref{problem1}. Haslinger and Roub\'i$\check{c}$ek \cite{haslinger1986optimal} first discussed the finite element approximations for the optimal control of variational inequality of the first kind. The article  \cite{haslinger1986optimal} shows convergence of the finite element approximation but does not talk about order of convergence. 

The article \cite{Meyer:2013:OCPOB} derived quasi-optimal a priori error estimates for Problem \eqref{problem1} using conforming finite element method based on certain assumptions on the mesh and given data. The analysis in \cite{Meyer:2013:OCPOB} is valid for convex $\Omega \subset \mathbb{R}^d$ where $d \in \{2,3\}$ with polygonal boundary $\partial \Omega$. The Discrete Maximum Principle (DMP) \cite{burman2004discrete} is a crucial tool to derive the improved $L^2$-error estimates for the obstacle problem. In \cite{Meyer:2013:OCPOB}, the DMP holds true in the sense that the conforming stiffness matrix is a weakly diagonally dominant M-matrix (see \cite{fiedler2008special}), which can be restrictive in applications. Later, \cite{dond2016nonconforming} proved the error estimates for the same model problem using the Crouzeix-Raviart (CR) nonconforming finite elements. This analysis is valid for bounded convex and polygon domain $\Omega \subset \mathbb{R}^2$. Furthermore, the  discretization is chosen in such a way that the quasi-uniform triangulation or a Friedrich-Keller triangulation with CR elements generates a global stiffness matrix which is a weakly domainated M-matrix, a slight variant of the DMP. This can again be challenging to fulfill in practice.

In the DG setting, there are only two works that studies the DMP for the Poisson problem, i.e., \cite{horvath2013discrete} and \cite{badia2015discrete}. The results in \cite{horvath2013discrete} are valid for dimension one, whereas \cite{badia2015discrete} proposes a new (variational) definition of the DMP in the DG setting for $d>1$. Here, we follow the approach that is discussed in \cite{
	badia2015discrete}. Our analysis is based on the quadratic growth property (Proposition \ref{grpp}) and the strong stationarity conditions (Theorem \ref{thm:main1}). The main contributions of our work are as follows: 

\smallskip
\noindent
$\bullet$ We do not assume/impose any restrictive assumption on our mesh as we have employed the definition for the local discrete maximum (respectively, minimum) and use the key property that DMP works in a local sense for DG methods, i.e., for all $K \in \mathcal{T}_h$, where $\cT_h$ is the triangulation. \\
$\bullet$  Derive the convergence rate of order $\frac{1}{2}- \epsilon$ for every $\epsilon> 0$ for the control $z$ 
when $\widehat\Omega = \Omega$. 
This result require $\Omega$ to be convex leading to $H^2(\Omega)$ regularity for the solution to the obstacle problem. \\
$\bullet$ In case of local matching state variable, we first derive novel quasi-optimal estimates for the state $\| u(z) - u_h(z) \|_{L^2(\Omega_0)} \lesssim h^{2-2\epsilon}\|u\|_{W^{2,p}(\Omega)}$. We use these estimates to derive quasi-optimal error estimates for control approximation of order $1-\epsilon$. These results require $u \in W^{2,p}$, for this regularity, we have assumed $\Omega$ to be $C^{1,1}$. Unfortunately we could not find a reference for this regularity for convex domains. 

The article is organized as follows: In Section \ref{sec2}, we begin with some preliminaries of Problem \eqref{problem1} and state the regularity estimates corresponding to the solution to the elliptic variational inequality. Next, we formulate the variational inequality into an equality with the help of continuous Lagrange multiplier $\zeta$ which acts as a function of both $\cV^*$ and $L^2(\Omega)$ in Section \ref{sec2}. In Section \ref{sec3}, we first highlight that the solution operator $S$ which maps control to state is not G\^ateaux-differentiable and later, we represent the Problem \eqref{problem1} into the mathematical problem with complementarity constraints (MPCCs) in function spaces. Using standard arguments, we state the existence and regularity result of an optimal control $\tilde{z}$ in Section \ref{sec3}. After that we state the strong stationarity conditions for any arbitrary obstacle using the ideas laid initially by Mignot in \cite{Mignot:1976:CEVI} for zero obstacle $g=0$. One of the key ingredient in our error analysis is the quadratic growth property satisfied by an optimal control $\tilde{z}$ (due to authors in \cite{Kunisch:2012:SOPS}) stated in Section \ref{sec3}. We introduce some notations and DG formulation in Section \ref{sec4}. Moreover, the matrix version of the strong stationarity conditions are discussed in Section \ref{sec4}. The main convergence results of our analysis require an improved a priori estimate for the $L^2$-error of the finite element discretization of the obstacle problem, which we have shown in Section \ref{sec5}. The idea is to use the Discrete Maximum Principle and $L^{\infty}$-error estimates for the Poisson equation which we have highlighted in Section \ref{sec5}. The a priori error estimates for $u$ and $z$ in both global and local matching cases are derived in Section \ref{sec5}.  Finally, conclusions and future directions are discussed in Section \ref{sec6}.

\section{Preliminaries and regularity results} \label{sec2}

The classical obstacle problem \cite{glowinski1984numerical,rodrigues1987obstacle} reads as:
\begin{subequations}
\begin{align}
a(&u, v-u) \geq (z, v-u) \quad \forall~ v \in \cK,  \label{obs3} \\
u &\in \cK=\{w \in \cV~:~w(x) \geq g(x)~\text{a.e. in } \Omega\}, \label{obs4}
\end{align}
\end{subequations}
where $(\cdot, \cdot)$ denotes the $L^2(\Omega)$ inner-product. In our analysis, $W^{m,p}(\Omega)$ 
denotes the Sobolev space equipped with the norm $\|\cdot\|_{W^{m,p}}$ and seminorm 
$|\cdot|_{W^{m,p}}$ \cite{s1989topics}. We assume that the given obstacle satisfies 
$g \leq 0 ~\text{a.e. on } \partial\Omega$, $g \in H^1(\Omega)$ and some other precise conditions on 
the given obstacle $g$ will be discussed later. The closed and convex set $\cK$ is non-empty because 
$g^+:= \max\{g,0\} \in \cK$. 
Next result collects the existence and regularity results of solution to \eqref{obs3}--\eqref{obs4} (see \cite[Lemma 2.1]{Kunisch:2012:SOPS} and \cite[Theorem~6.2]{DKindelehrer_GStampacchia_1980}, \cite[Chapter 5, Corollary 2.3 and Theorem 2.5]{rodrigues1987obstacle}).
\begin{theorem}[Existence and regularity of the obstacle problem] \label{asp1} 
	The following holds:
	\begin{enumerate}[(a)]
		\item If $\Omega \subset \mathbb{R}^d$ is a polygonal domain, $g \in H^1(\Omega)$, and $z \in H^{-1}(\Omega)$, 
			then there exists a unique solution to \eqref{obs3}-\eqref{obs4} satisfying $u \in H^1_0(\Omega)$.
		\item If $\Omega$ is convex or $\partial \Omega$ is $C^{1,1}$, $g \in H^2(\Omega)$, and $z \in L^2(\Omega)$, 
		then the solution to \eqref{obs3}-\eqref{obs4} satisfies $u \in H^2(\Omega) \cap H^1_0(\Omega)$.
		\item Let $\partial \Omega$ be $C^{1,1}$, $z \in L^p(\Omega) \cap H^{-1}(\Omega)$ and 
			$g \in W^{2,p}(\Omega)$, with $1 < p < + \infty$ then the solution to \eqref{obs3}-\eqref{obs4}  
		satisfies $u \in W^{2,p}(\Omega) \cap H^1_0(\Omega)$.
	\end{enumerate}
\end{theorem}
\begin{remark}
	In Theorem \ref{asp1}(c), note that $L^p(\Omega) \subset  H^{-1}(\Omega)$ for 
		$d=2$ when $1 < p< + \infty $ and $L^p(\Omega) \subset  H^{-1}(\Omega)$
		for $d=3$ when $\frac{6}{5} \le p < +\infty$. These conditions are always
		satisfied by our control $z$ which is at least $L^2(\Omega)$.

\end{remark}
	In case of bounded Lipschitz domain $\Omega$, there exists a unique
	Lagrange multiplier $\zeta \in \cV^*$ (the dual of $\cV$) \cite{veeser2001efficient} such that  
	\begin{align}
	a(u,v) &= ( z, v) + \langle \zeta , v \rangle_{-1,1} \quad \forall~v \in \cV, \label{eq:C4} \\
	\langle \zeta, u- g \rangle_{-1,1} =0 &~, ~\zeta \geq 0 ~\text{and } u(x) - g(x)\geq 0~~\text{ a.e. in } \Omega, \label{eq:C5}
	\end{align}
where $\langle \cdot, \cdot \rangle_{-1,1}$ denotes the duality pairing between $\cV$ and $\cV^*$. The equation
\eqref{eq:C5} is known as the complementarity equation with $\zeta \geq 0$ on $\Omega' \subseteq \Omega$ if $\langle \zeta, \phi \rangle_{-1,1} \geq 0$, for all $0 \leq \phi \in H^1_0(\Omega')$.

	We can also equivalently rewrite the bilinear form $a(\cdot,\cdot)$ as follows. 
	We have $\cA \in \cL(\cV, \cV^*)$ i.e., the set of all bounded linear operators from $\cV$ to $\cV^*$ such that
\begin{align} \label{keyprop}
a(v,w)= \langle \cA v, w \rangle_{-1,1} \qquad \forall~v ,w \in \cV.
\end{align}
For any $f \in \cV^*$, the norm $\|f\|_{\cV^*}$ is defined as
\begin{align} 
\|f\|_{\cV^*} := \underset{v  \in \cV,~ v \neq 0}{\text{sup}} \frac{\langle f, v \rangle_{-1,1}}{\|\nabla v\|_{L^2(\Omega)}}.
\end{align}

\section{The strong stationarity conditions} \label{sec3}
	The goal of this section is to discuss the strong stationarity conditions for the 
	model Problem~\eqref{problem1}. The content of this section is known in the literature
	especially for the case $g \equiv 0$ (see \cite{Mignot:1984:OCPVI,Mignot:1976:CEVI}).
	We first recall the notion of the differentiability in Hilbert spaces \cite{HA:2018:PDEOCP}.
\begin{definition}
	Let $X$ and $Y$ be two Banach spaces and $\cL:=\cL(X,Y)$ denotes the space of all bounded linear operators from $X$ to $Y$. Let $X'$ be an open subset of $X$, then the map $J: X' \rightarrow Y$ is said to be G\^ateaux differentiable at $z \in X'$ if it is directionally differentiable in all directions $t \in X$ and $J'(z,t)=J'(z) t$ where $J'(z) \in \cL(X,Y)$.
\end{definition}
\begin{remark} \label{rem:1}
	The solution map $S: L^2(\Omega) \rightarrow \cK \subseteq \cV$ 
	which is defined as $S(z):=u \in \cK$ is Lipschitz continuous (see Lemma 2.1 in \cite{Kunisch:2012:SOPS}) but in general is not G\^ateaux-differentiable (see Lemma 2.4 in \cite{Kunisch:2012:SOPS}). Whereas, $S$ is directionally differentiable at all $z \in L^2(\Omega)$. Using Theorem 3.3 of article \cite{Mignot:1976:CEVI}, the directional derivative given by $DS(z,t) \in \cM_{u}$ in the direction $t \in \cV^*$ satisfies the following variational inequality
	\begin{align} \label{1.4}
	a(DS(z,t),v-DS(z,t)) \geq \langle t, v -DS(z,t) \rangle_{-1,1} \quad \forall~ v \in \cM_{u},
	\end{align}  
	where $\cM_{u} := \big\{v \in \cV~:~ v(x) \geq 0~\text{ a.e. on } \{u-g=0\} \text{ and } \langle \zeta, v \rangle_{-1,1} =0\big \}$. 
\end{remark}	
	
	Using \eqref{eq:C4} and \eqref{eq:C5}, Problem \eqref{problem1} can be rewritten as the mathematical program with complementarity constraints (MPCCs) in function spaces which is a particular case of the mathematical program with equilibrium constraints (MPECs) \cite{Mignot:1984:OCPVI}. For notation simplicity, we will focus on $\widehat{\Omega} = \Omega$, with almost no changes to the text for $\widehat\Omega = \Omega_0$.
	\begin{problem}
		\begin{subequations}
			\begin{align}
		\min_{z \in L^2(\Omega)} \left\{ F(z):= \frac{1}{2} \|S(z)-u_d\|^2_{L^2(\Omega)} + \frac{\nu}{2} \|z\|^2_{L^2(\Omega)} \right\}, \label{eq:C1} \\  \text{ subject to } \quad a(u,v) = (z,v) + \langle \zeta , v \rangle_{-1,1} \quad \forall ~v \in \cV, \label{eq:C2} \\
		\langle \zeta, u- g \rangle_{-1,1} =0 \, , ~\zeta \geq 0 ~\text{and } u(x) - g(x)\geq 0~~\text{ a.e. in } \Omega.  \label{eq:C3}
		\end{align}
		\end{subequations}
		\end{problem}
	Existence of solution to the above optimization problem follows by standard arguments. However, in general we cannot expect $\tilde{z}$ to be unique.
	
The strong stationarity conditions have been stated in the next result. See \cite[Theorem 2.2]{Mignot:1984:OCPVI} for 
the proof when $g=0$ and for $g\neq 0$, the proof follows using similar arguments as in \cite{Mignot:1976:CEVI,Mignot:1984:OCPVI} provided optimal state satisfies $\tilde{u} \in H^2(\Omega)$ 
(via Theorem \ref{asp1}(b)) or $g|_{\partial\Omega} = 0$.
	\begin{theorem}[Strong stationarity conditions]
\label{thm:main1}
	Let $\tilde{u} \in H^2(\Omega)$ or $\tilde{u} \in \cK$ with $g|_{\partial \Omega}=0$ and let $(\tilde{z}, \zeta) \in L^2(\Omega) \times \cV^*$ satisfy the equations \eqref{eq:C1}--\eqref{eq:C3}. Then, there exists a slack variable $\lambda \in \cV^*$, an adjoint variable $p \in H^1_0(\Omega)$ such that $(\tilde{u}, \tilde{z}, p, \lambda)$ solves the following strong stationarity conditions:
	\begin{subequations}\label{eq:strong_stat}
	\begin{align}
	a(\tilde{u},v) &= \int_{\Omega} \tilde{z}~ v~ dx + \langle \zeta , v \rangle_{-1,1}~~~\forall ~v \in \cV, \label{1.5a} \\ \langle \zeta, \tilde{u}- g \rangle_{-1,1} &=0 \, ,~\zeta \geq 0 ~\text{and } \tilde{u}(x) - g(x)\geq 0~~\text{ a.e. in } \Omega, \label{1.5b} \\ a (v, p) &= \int_{\Omega} (\tilde{u} - u_d) v~dx + \langle \lambda , v\rangle_{-1,1}~~~\forall ~v \in \cV,  \label{1.5c}\\  \langle \lambda, v \rangle_{-1,1} &\leq 0 ~~\forall~ v \in \cM_{\tilde{u}} ~ \text{ and }~ p \in \cM_{\tilde{u}},  \label{1.5d} \\ p(x) &+ \nu \tilde{z}(x) = 0 ~\text{ a.e. in } \Omega , \label{1.5e}
	\end{align}
	\end{subequations}
where $\cM_{u} := \big\{v \in \cV~:~ v(x) \geq 0~\text{ a.e. on } \{u-g=0\} \text{ and } \langle \zeta, v \rangle_{-1,1} =0\big \}$.	
\end{theorem}	

From \eqref{1.5e}, using the regularity of adjoint $p$, we immediately have that 
$\tilde{z} \in H^1_0(\Omega)$. The embedding of $H^1(\Omega)$ in $L^p(\Omega)$
is standard \cite{s1989topics}. 
	\begin{proposition}\label{prop:ctrl_reg}
		Under the assumptions of Theorem~\ref{thm:main1}, optimal control 
		$\tilde{z} \in L^2(\Omega)$ corresponding to Problem \eqref{problem1} fulfills 
		\[ 
		\tilde{z} \in 	H^1(\Omega) \hookrightarrow
		\begin{cases}
		L^6(\Omega) \quad \text{if}~ d =3,  \\  L^p(\Omega)\quad \text{if}~ d =2,
		\quad 		\mbox{where } 1 \leq p < \infty.
		\end{cases} 
		\] 
	\end{proposition}

The optimal control $\tilde{z} \in L^2(\Omega)$ of \eqref{eq:C1} also satisfies the {\it quadratic growth condition} which is stated next and used in the subsequent analysis. For its proof, see \cite[Thm.~2.12]{Kunisch:2012:SOPS}.

\begin{proposition} \label{grpp}
	Let $F(\cdot)$ be the functional defined in \eqref{eq:C1}, then there exists constants
	$\gamma>0$ and $\kappa >0$ such that
	\begin{align} \label{grpp1}
	F(\tilde{z}) \leq F(z) - \kappa \|z-\tilde{z}\|_{L^2(\Omega)}^2 \quad \forall~ z \in B_{\gamma}(\tilde{z}),
	\end{align}
	where $B_{\gamma}(\tilde{z})$ is the ball of radius $\gamma$ with center $\tilde{z}$ in the topology of $L^2(\Omega)$.
\end{proposition}
\section{Discretization} \label{sec4}
\noindent
Let $\cT_h$ be a regular triangulation (mesh) of $\Omega$ and $K \in \cT_h$ denotes a non-degenerate element (triangle) for $d=2$ or tetrahedron for $d=3$.
For the discretization of Problem \eqref{problem1}, we use a discontinuous Galerkin finite element space given by
\begin{align} \label{dgspc}
\cV_h:= \{v_h \in L^2(\Omega)~:~v_h|_K \in \cP_1(K)~ \forall~ K \in \cT_h\},
\end{align} where $w|_K$ denotes the restriction of $w$ to $K$ and $\cP_1(K)$ is the space of polynomials of degree less than or equal to one defined on $K$. Let $\{z_i~:~ i=1,2, \dots, M_h\}$ denotes the set of all vertices (corner points) of the mesh, where $M_h$ is the total number of vertices and the corresponding macroelements (patch) will be denoted by $\omega_i=\cup_{z_i \in K} K$. The discrete function $v_h \in \cV_h$ can be expressed as the linear combination of the Lagrange basis functions $\phi^K_i$ where $(i,K) \in \{1,2,\cdots, M_h\} \times \cT_h$. Note that $\phi^{K}_i(z_j)=\delta_{ij}$, where $\delta_{ij}$ is the Kronker's delta function and $\phi^K_i(x)=0~\forall x \in \Omega \setminus K$. Therefore, any $v_h \in \cV_h$ can have the following form:
\begin{align} \label{disfuc}
v_h=\sum_{i=1}^{M_h} \sum\limits_{ K \subset \omega_i} v_{h,i}^K \phi^K_i.
\end{align} 
Let $\cV_K$ be the set of all nodes of element $K \in \cT_h$. Further, $\mathcal{E}_h$ is the set of all the edges/faces of $\cT_h$ and $h_e$ is the length of an edge/face $e$. Let us introduce the discrete nonempty, closed and convex subset $\cK_h$ as follows:
\begin{align}
\cK_h:= \{w_h \in \cV_h~:~w_h|_K(p) \geq g_h(p) ~ \forall~ p \in \cV_K~\text{and } \forall~ K \in \cT_h\},
\end{align}where $g_h \in \cV_c:= \cV_h \cap \cV$ denotes the nodal interpolation of $g$ \cite{BScott:2008:FEM}. Before introducing the discrete variational inequality, we need to introduce the average and jump of the discrete functions. Firstly, the broken Sobolev space is defined by
\begin{align*}
H^1(\Omega, \cT_h)=\{v \in L^2(\Omega)~:~v|_K \in H^1(K) \quad \forall~K \in \cT_h\}.
\end{align*}
Let $e$ be an interior edge/face in $\cT_h$, then there exist two elements $K^+$ and $K^-$ such that $e = \partial K^+ \cap \partial K^-$. Let $n^+$ be an unit outward normal pointing from $K^+$ to $K^-$, then we have $n^-=-n^+$. Hence, the average and jump of $v \in H^1(\Omega, \cT_h)$ on an edge/face $e$ is defined by
\begin{align*}
\smean{v}= \frac{v^++v^-}{2} \quad \text{and} \quad  \sjump{v}= v^+n^+ + v^-n^-,
\end{align*} respectively where $v^+=v|_{K^+}$ and $v^-=v|_{K^-}$. Similarly, we define the average and jump for a vector valued function $q \in [H^1(\Omega, \cT_h)]^d$ on interior edge/face $e$ as
\begin{align*}
\smean{q}= \frac{q^++q^-}{2}\quad \text{and} \quad \sjump{q}= q^+\cdot n^+ + q^-\cdot n^-.
\end{align*} 
Let $\cE^b_h$ be the set of all boundary edges of $\cT_h$, then for a boundary edge $e \in \cE^b_h$ let $n^e$ be an outward unit normal to an element $K$ such that $\partial K \cap \partial \Omega =e$, we define for $v \in H^1(\Omega, \cT_h) $
\begin{align*}
\sjump{v}=v n^e \quad \text{and} \quad \smean{v} = v
\end{align*} and for $q \in [H^1(\Omega, \cT_h)]^d$, we set
\begin{align*}
\sjump{q}=q \cdot n^e \quad \text{and} \quad \smean{q}= q.
\end{align*}
Next, let us define the following two bilinear forms
\begin{equation*}
a_h(v_h,w_h):= (\nabla_h v_h, \nabla_h w_h) = \sum_{K \in \cT_h}\int_K \nabla v_h \cdot \nabla w_h~ dx,
\end{equation*}
and for all $w , v \in \cV_h$
\begin{align}
b_h(w,v) := - \sum_{e \in \cE_h} \int_e \smean{w} \sjump{v}~ ds  & - \sum_{e \in \cE_h} \int_e \smean{\nabla v} \sjump{w}~ ds + \sum_{e \in \cE_h} \int_e \frac{\eta}{h_e} \sjump{w}\sjump{v}~ ds,
\end{align}
where $\eta \geq \eta_0 >0$ is sufficiently large positive number to ensure the ellipticity of $\mathcal{A}^{SIP}(\cdot,\cdot)$ which is defined by
\begin{equation} \label{DP12}
\mathcal{A}^{SIP}(v_h,w_h)=a_h(v_h,w_h) + b_h(v_h,w_h) ~\forall v_h,w_h \in \cV_h.
\end{equation} 
The discrete version of the obstacle problem in Problem \eqref{problem1} is given by: 
Find $u_h \in \cK_h$ fulfilling
\begin{align} \label{disobs}
\xA^{SIP}(u_h, v_h-u_h) \geq (z, v_h-u_h)~ \forall~  v_h \in \cK_h  ,
\end{align}
where $\mathcal{A}^{SIP}(\cdot,\cdot): \cV_h \times \cV_h \rightarrow \mathbb{R}$ is the SIPG (Symmetric Interior Penalty Galerkin) bilinear form \cite[Chapter 10]{BScott:2008:FEM}. Using the coercivity and boundedness of the discrete bilinear form $\xA^{SIP}(\cdot,\cdot)$ \cite{WHC:2010:DGOP}, it can be verified that the discrete problem \eqref{disobs} has a unique solution $u_h \in \cK_h$.

Next, we introduce the discrete solution map corresponding to $u_h \in \cK_h \subseteq \cV_h$ as follows:
\begin{align*} 
S_h : L^2(\Omega) \rightarrow \cV_h \quad, \quad z \mapsto S_h(z)=: u_h. 
\end{align*}
The discrete version of the optimization problem \eqref{eq:C1} using $S_h$ is stated next:
\begin{align} 
\min_{z \in L^2(\Omega)} \bigg\{ F_h (z)&:= \frac{1}{2} \|S_h(z)-u_d\|^2_{L^2(\Omega)} + \frac{\nu}{2} \|z\|^2_{L^2(\Omega)} \label{disocp} \\ 	&=  \frac{1}{2} \sum_{K \in \cT_h} \|S^K_h(z)-u_d\|^2_{L^2(K)} + \frac{\nu}{2} \|z\|^2_{L^2(\Omega)} \bigg\}. \notag 
\end{align}
As in the continuous case, the standard arguments yield existence of local optimal solution to  \eqref{disocp}. In \eqref{disocp}, the control $z \in L^2(\Omega)$ is not discretized. However, we will see in Remark \ref{key1} that each local optimum $z$ of \eqref{disocp} belongs to $\cV_h$ (defined in \eqref{dgspc}) and there is no need to discretize the control variable. 
To begin, we introduce the space 
\begin{align*}
\cW:=\cV_h + (\cV \cap H^2(\Omega)),
\end{align*} where $ A + B := \{a + b~:~a \in A~\text{and }~ b \in B \}$ and the norm $\vertiii{\cdot}$ for $w \in \cW$ is defined by
\begin{align}
\vertiii{w}^2:= \sum\limits_{ K \in \cT_h} \|\nabla w\|^2_{L^2(K)} + \sum\limits_{ e \in \cE_h} \frac{1}{h_e}\|\sjump{w}\|^2_{L^2(e)} + \sum\limits_{ K \in \cT_h} h_K^2\|D^2w\|^2_{L^2(K)} ,
\end{align}
where $h_K:=$ diameter of $K$. It is easy to see that $\vertiii{w}$ is a norm on $\cW$ with the help of next lemma. We deduce that the $\| \cdot\|_{L^2(\Omega)} $ is bounded by $\vertiii{\cdot}$ through the next result. For more details, see \cite[Lemma 2.1]{arnold1982interior}. We refer to Appendix \ref{sec7} for a proof.
\begin{lemma} \label{normeq}
	It holds that
	\begin{align}
	\|w\|_{L^2(\Omega)} \leq C_1 \bigg(\sum\limits_{ K \in \cT_h} \|\nabla w\|^2_{L^2(K)} + \sum\limits_{ e \in \cE_h} \frac{1}{h_e}\|\sjump{w}\|^2_{L^2(e)} \bigg)^{\frac{1}{2}} \quad \forall~ w \in H^1(\Omega, \cT_h),
	\end{align}
	where $C_1$ is a positive constant. 
\end{lemma}
In the following result, we discuss the Lipschitz continuity of $S_h$. 
\begin{lemma}[$S_h$ is Lipschitz]
	For every given $z^1 ,z^2 \in L^2(\Omega)$, there exist unique solutions $S_h(z^1):= u^1_h \in \cK_h$ and $S_h(z^2):=u^2_h\in \cK_h$ of \eqref{disobs} and the following holds:
	\begin{align} \label{lipzSh}
	\|S_h(z^1)- S_h(z^2)\|_h \leq \|z^1-z^2\|_{L^2(\Omega)},
	\end{align}
	where $\|w\|_h^2:= \sum\limits_{ K \in \cT_h} \|\nabla w\|^2_{L^2(K )} + \sum\limits_{ e \in \cE_h} \frac{1}{h_e}\|\sjump{w}\|^2_{L^2(e)}$ is the norm on $\cV_h$ \cite{WHC:2010:DGOP}.
\end{lemma}
\begin{proof}
	We prove \eqref{lipzSh} for every $K \in \cT_h$. We begin the proof by letting $S^{K}_h(z^1)= u^1_h~;~S^{K}_h(z^2)= u^2_h$ for some ${K} \in \cT_h$. Next, we insert $v_h=u^1_h \in \cK_h$ for $z = z^2$ in \eqref{disobs} and $v_h=u^2_h \in \cK_h$ for $z = z^1$ in \eqref{disobs} and we observe the following:
	\begin{align}
	\xA^{SIP}(u^2_h, u^1_h-u^2_h) &\geq ( z^2, u^1_h-u^2_h), \label{eq1}\\
	\xA^{SIP}(u^1_h, u^2_h-u^1_h) &\geq (z^1, u^2_h-u^1_h). \label{eq2}
	\end{align}
	We employ the coercivity of $\xA^{SIP}(\cdot, \cdot)$, 
	Lemma \ref{normeq} and add the equations \eqref{eq1}--\eqref{eq2} to have the following:
	\begin{align*}
	\|u^1_h-u^2_h\|^2_h \leq \xA^{SIP}(u^1_h-u^2_h,u^1_h-u^2_h) \leq \|z^1-z^2\|_{L^2(\Omega)} \|u^1_h-u^2_h\|_h.
	\end{align*}
	Finally, we conclude the proof of \eqref{lipzSh}.
\end{proof}
Next, we introduce the discrete Lagrange multiplier $\zeta_h \in \cV_h$ as
\begin{align}
\langle \zeta_h , v_h \rangle_h= \xA^{SIP}(u_h,v_h) - \int_{\Omega} z v_h~ dx  \quad \forall~ v_h \in \cV_h, \label{dsclag}
\end{align} which could be thought as an approximation to the continuous functional $\zeta$ stated in \eqref{eq:C4} and the inner product $\langle \cdot, \cdot \rangle_h$ is defined by: for any $v_h, w_h \in \cV_h$,
\begin{align}
\langle v_h, w_h \rangle_h:= \sum_{K \in \cT_h} \frac{|K|}{d+1} \sum_{p \in \cV_K} v_h(p)w_h(p).
\end{align}
Using the definition \eqref{dsclag} and discrete variational inequality \eqref{disobs}, we get that
\begin{align}
\langle \zeta_h , v_h-u_h \rangle_h \geq 0 \quad \forall~v_h \in \cK_h. \label{propslag}
\end{align}Let us suppose $\phi^p_h \in \cV_h$ be the canonical Lagrange basis function associated with the vertex $p$, i.e., $\phi^p_h$ takes the value one at the vertex $p$ and vanishes at all other vertices. We note that $\zeta_h|_K \in \cP_1(K)$ and by choosing $v_h = u_h + \phi^p_h \in \cK_h$ in \eqref{propslag}, we obtain that $\forall~K \in \cT_h$
\begin{align}
\zeta_h \geq 0~\text{in } K \iff \zeta_h(p) \geq 0~\forall~ p \in \cV_K.
\end{align} 
The finite element approximation for the stationarity conditions mentioned in Theorem \ref{thm:main1} can be stated as follows:
\begin{proposition} \label{disfor}
	Let $\tilde{z}_h \in \cV_h$ be a local optimum of \eqref{disocp}, $\tilde{u}_h \in \cV_h$ be the discrete state variable, $\tilde{p}_h \in \cV_h$ be the discrete adjoint variable and lastly $\lambda_h \in \cV_h$ be the discrete slack variable. Then, it holds that
	\begin{subequations}
		\begin{align}
		\xA^{SIP}(\tilde{u}_h,v_h) &= \int_{\Omega} \tilde{z}_hv_h~ dx + \langle \zeta_h , v_h \rangle_h \quad \forall~ v_h \in \cV_h, \label{1.8a} \\ \langle \zeta_h, \tilde{u}_h- g_h \rangle_h &=0 ~;~\zeta_h \geq 0~ \forall~ K \in \cT_h,  \label{1.8b} \\  \hspace*{-1cm}& (\tilde{u}_h - g_h)(p) \geq 0 ~~\forall~ p \in \cV_K \text{ and } \forall~ K \in \cT_h,\\ \xA^{SIP}(v_h, p_h) &= \int_{\Omega} (\tilde{u}_h - u_d) v_h~dx + \langle \lambda_h , v_h \rangle_h~~~\forall ~v_h \in \cV_h, \label{1.8c}\\  \langle \lambda_h, v_h \rangle_h &\leq 0 ~~\forall~ v_h \in \cM_{\tilde{u}_h} ~ \text{ and }~ p_h \in \cM_{\tilde{u}_h}, \label{1.8d} \\ p_h(x) &+ \nu \tilde{z}_h(x) = 0 ~\text{ a.e. in } \Omega.\label{1.8e}
		\end{align}
	\end{subequations}
	where,
	\begin{align*}
	\cM_{\tilde{u}_h} := \big\{v_h \in \cV_h~:~ v_h(x) \geq 0~\text{ a.e. on } \{\tilde{u}_h-g_h=0\} \text{ and } \langle \zeta_h, v_h \rangle_h =0\big \}.
	\end{align*}
\end{proposition}
For any $K \in \cT_h$, let $\{\phi^K_p~;~p \in \cV_K\}$ be its local Lagrange basis functions, then we set $v_h^K=v_h|_K$. Using this notation, we define $v_h$ on the single element $K$ as
\begin{align} \label{notation}
v_h^K:= \sum_{p \in \cV_K} v_h(p) \phi^K_p.
\end{align}
Then, we can rewrite \eqref{disobs} in the following way using the above notation \eqref{notation} and support of the local basis functions $\phi^K_p$, i.e., $\phi^K_p(x) = 0\quad \text{if } x \in \Omega \setminus K$.
\begin{align}
\sum_{p,q \in \cV_K} u_h(p) \xA^{SIP}( \phi^K_p,\phi^K_q ) (v_h(q)-u_h(q)) & \geq \sum_{q\in \cV_K} (z, \phi^K_q) (v_h(q)-u_h(q)) \quad \forall~ K \in \cT_h, \notag \\ \hspace{-1cm}\iff  (^\mathsf{\top}\textbf{u}_{h,K}) \mathbf{\xA}^{SIP}_K (\textbf{v}_{h,K}-\textbf{u}_{h,K}) & \geq (^\mathsf{\top}\textbf{b}_{z,K}) (\textbf{v}_{h,K}-\textbf{u}_{h,K})~ \forall~ \textbf{v}_{h,K} \in \mathbb{R}^{|\cV_K|}  \label{disvar} \\ & \hspace*{-2cm} \text{and } \textbf{v}_{h,K}-\mathbf{g}_{h,K} \geq 0 \quad \forall~ K \in \cT_h, \notag
\end{align}
where \begin{enumerate}
	\item $\mathbf{\xA}^{SIP}_K\in \mathbb{R}^{|\cV_K| \times |\cV_K|}~\text{having the } (p,q) \text{ entry as } \xA^{SIP}( \phi^K_p,\phi^K_q )$.
	\item $\textbf{b}_{z,K}:= (z, \phi^K_q)_{q \in \cV_K} \in \mathbb{R}^{|\cV_K|}$.
	\item $\textbf{w}_{h,K}:= (w_h(p))_{p \in \cV_K}~\forall~ w_h \in \cV_h$.
	\item $^\mathsf{\top}\textbf{a}:=$ Transpose of the vector $\textbf{a} \in \mathbb{R}^{|\cV_K|}$.
	\item Let $w_h, v_h \in \cV_h$, then $\textbf{w}_{h,K}-\mathbf{y}_{h,K} \geq 0$ is understood in the sense that $w_h(p) -y_h(p) \geq 0 ~\forall~p \in \cV_K$.
\end{enumerate}

\subsection{The Matrix Formulation}
We revisit the matrix version of the Lagrange multiplier in the following way:
\begin{align*}
\mathbf{\zeta}_K:= \mathbf{\xA}^{SIP}_K\textbf{u}_{h,K} - \textbf{b}_{z,K} \in \mathbb{R}^{|\cV_K|} \quad \text{for every } K \in \cT_h.
\end{align*}
In the view of the definition of $\mathbf{\zeta}_K$, we reformulate \eqref{disvar} as:
\begin{align}
&\mathbf{\xA}^{SIP}_K\mathbf{u}_{h,K}  = \mathbf{\zeta}_K +  \textbf{b}_{z,K},  \label{disc1}\\& u_h(p) \geq g_h(p) ~\forall~p \in \cV_K, \quad \mathbf{\zeta}_K \geq 0, \quad (^\mathsf{\top}\mathbf{\zeta}_K)(\textbf{u}_{h,K}-\mathbf{g}_{h,K})=0. \label{disc2}
\end{align}
For each element $K \in \cT_h$, we have the equivalent mathematical problem with complementarity constraints (MPCC) to \eqref{disocp}. For some $K \in \cT_h$, we define the minimum function $\text{min}(\textbf{a}, \textbf{b})$, where minimum is understood in the componentwise sense
\begin{equation} \label{problem2}
\begin{cases}
\min \quad \quad \frac{1}{2} (^\mathsf{\top}\textbf{u}_{h,K}) \mathbf{\cN} \textbf{u}_{h,K} - (^\mathsf{\top}\textbf{u}_{h,K}) \textbf{b}_d + \frac{\nu}{2} \|z\|^2_{L^2(\Omega)} & \\ \text{subject to}~~  \mathbf{\xA}^{SIP}_K\textbf{u}_{h,K}  = \mathbf{\zeta}_K +  \textbf{b}_{z,K} \quad \text{and} \quad \min(\textbf{u}_{h,K}-\mathbf{g}_{h,K}, \mathbf{\zeta}_K) = 0 , 
\end{cases}
\end{equation}
where \begin{enumerate}
	\item For every $K \in \cT_h$, the mass matrix $\cN$ is defined as
	\begin{align}
	\cN \in \mathbb{R}^{|\cV_K| \times |\cV_K|}~;~\cN_{pq} := \int_K \phi^K_p \phi^K_q dx~ \forall p,q \in \cV_K.
	\end{align} 
	\item For all $p \in \cV_K$ and $K \in \cT_h$, we define $\textbf{b}_d: = \int_K u_d \phi^K_p dx$.
\end{enumerate}
The matrix version of the finite element approximation given in Proposition~\ref{disfor} is discussed in the next theorem. The proof of the next result follows along the same as in theory of finite dimensional MPECs \cite{Bergoun:2000:OCPOB}. Here, we use $\odot$ to indicate the Hadamard product.
\begin{theorem} \label{thm:main2}
	Let $\tilde{z}_h:=(\tilde{z}^K_{h})_{K \in \cT_h} \in L^2(\Omega)$ be a discrete optimal solution to \eqref{problem2} having an associated state variable $\tilde{u}_h:=(\tilde{u}^K_{h})_{K \in \cT_h} \in \cV_h$ and a Lagrange multiplier (matrix version) $\mathbf{\tilde{\zeta}}_K:= (\tilde{\zeta}_{K,i})^3_{i=1} \in \mathbb{R}^{|\cV_K|}$. Therefore, we have existence of a discrete adjoint state variable $\tilde{p}_h:=(\tilde{p}^K_{h})_{K \in \cT_h}  \in \cV_h$ and a discrete multiplier $\mathbf{\lambda}_K:=(\lambda_{K,i})^3_{i=1} \in \mathbb{R}^{|\cV_K|} $ such that the following holds $\forall ~K \in \cT_h$:
	\begin{subequations}
		\begin{align}
		\mathbf{\xA}^{SIP}_K \mathbf{\tilde{u}}_{h,K}  &= \mathbf{\tilde{\zeta}}_K +  \textbf{b}_{\tilde{z}_{h},K}, \\ (^\mathsf{\top}\mathbf{\tilde{\zeta}}_K) (\mathbf{\tilde{u}}_{h,K}-\mathbf{g}_{h,K})& =0~;~\mathbf{\tilde{\zeta}}_K \geq 0 \text{ and } \mathbf{\tilde{u}}_{h,K}-\mathbf{g}_{h,K} \geq 0, \\ (^\mathsf{\top}\mathbf{\xA}^{SIP}_K) \mathbf{\tilde{p}}_{h,K} &= \cN\mathbf{\tilde{u}}_{h,K} - \textbf{b}_d + \mathbf{\lambda}_K,\\ (\mathbf{\tilde{u}}_{h,K}-\mathbf{g}_{h,K}) \odot \mathbf{\lambda}_K &= \mathbf{0}~;~ \mathbf{\tilde{\zeta}}_K \odot \mathbf{\tilde{p}}_{h,K}  = \mathbf{0},  \\ \lambda_{K,i} \leq 0 ~\forall~ i \in \{1,2,3\}~&;~\tilde{p}_h(q) \geq 0 ~\forall~ q \in \cV_K \text{ along with } \tilde{u}_h(q)- g_h(q) = 0= \tilde{\zeta}_{K,q}, \\ \nu \tilde{z}^K_{h}(x) + \tilde{p}^K_h (x) & = 0 \quad \text{a.e. in } K. \label{obs2}
		\end{align}
	\end{subequations}
\end{theorem}
\begin{remark} \label{key1}
	With the help of \eqref{obs2} and if $\tilde{z}_h \in L^2(\Omega)$ is a discrete optimal solution to \eqref{disocp}, then $\tilde{z}_h \in \cV_h$.
\end{remark}
\section{The Error Analysis} \label{sec5}
In this section, we discuss the a priori error analysis for the symmetric interior penalty discontinuous Galerkin method for the optimal control problem governed by the elliptic obstacle problem (Problem \eqref{problem1}). To prove the error estimates for the state (local matching) and control variables, we derive the improved $L^2$-error estimates for the obstacle problem  for $\cP_1$ DG FEM discretization. We also recall the energy norm estimates for the elliptic variational inequality \cite{WHC:2010:DGOP} for the state (global matching) and control variables.
\subsection{Error estimates for the obstacle problem}
For a given control $z \in L^2(\Omega)$, we consider the following obstacle problem:
\begin{align}
a(u, v-u) &\geq (z, v-u) \quad \forall~ v \in \cK,  \label{obs} \\
u &\in \cK=\{w \in \cV~:~w(x) \geq g(x)~\text{a.e. in } \Omega\}. \label{obs1}
\end{align}
The next theorem recalls the energy norm error estimates for the obstacle problem from \cite{WHC:2010:DGOP}.
See Theorem 4.1 on Pg. 718 in \cite{WHC:2010:DGOP} for a proof. We recall the regularity result for the obstacle problem from Theorem~\ref{asp1} which has been used below.
\begin{theorem} \label{main:thm1}
	Let $u$ and $u_h$ be the continuous and discrete solutions of \eqref{obs} and \eqref{disobs}, respectively. Then, 
	there is a positive constant $C$ depending on the penalty parameter $\eta>0$ such that the following error estimate 
	holds
	\begin{align}
	\vertiii{u-u_h} \leq Ch\big(\|u\|_{H^2(\Omega)} + \|g\|_{H^2(\Omega)}\big), \label{err1}
	\end{align}
	where $h:= \max \{h_K~:~K \in \cT_h\}$.
\end{theorem}

With the help of Theorem \ref{main:thm1} and Lemma \ref{normeq}, we have the estimate \eqref{err2}.          
\begin{remark} \label{rem1}
	Let $C >0$ be a positive constant which depends on the penalty parameter $\eta>0$, the domain $\Omega$ and on the given bounds $\|u\|_{H^2(\Omega)} $ and $\|g\|_{H^2(\Omega)}$. Then, there holds
	\begin{align}
	\|u-u_h\|_{L^2(\Omega)} \leq Ch^{\alpha}. \label{err2}
	\end{align}	
	Note that, $\alpha =1$ is evident from Theorem \ref{main:thm1} but the proof for $\alpha >1$ is much more delicate. This will be the focus of subsection \ref{sub5.2} and it is one of the main novel aspect of this paper. The celebrated Aubin-Nitsche trick from the theory of elliptic PDEs does not directly apply to the obstacle problem \cite{natterer1976optimale}. To overcome this difficulty, we will employ the $L^{\infty}$-error estimates for the obstacle problem based on the Discrete Maximum Principle (DMP) for DG methods \cite{badia2015discrete}. The initial ideas were laid by Nitsche \cite{nitsche1975l_} which are based on the DMP. The key difference is that we need $W^{2,p}$- regularity on the solution of obstacle problem (Theorem \ref{asp1}) whereas \cite{nitsche1975l_} requires $W^{2, \infty}$-regularity. 
\end{remark}
\subsection{Improved error estimates for the obstacle problem} \label{sub5.2}
	We adopt the following definition of local extremum in the DG setting and 
	refer to article \cite{badia2015discrete} for more details.
\begin{definition} \label{ext}
	Let $v_h \in \cV_h$ and $K \in \cT_h$, then $v_h$ has a local discrete minimum (resp., maximum) at a node $x_i \in K$ if $v^K_h(x_i) \leq v_h(z)$ (resp., $v^K_h(x_i) \geq v_h(z)$) $\forall~z \in \omega_i= \bigcup_{x_i \in K} K$.
\end{definition}
Let us consider the index set $\mathcal{N}:= \{1,2, \cdots , N+M\}$ to be the set of all corner points of $\cT_h$, where $\{x_i\}_{i=1}^N$ and $\{x_i\}_{i=N+1}^M$ denote the total number of interior and boundary corner points of $\cT_h$, respectively. Next, using the above Definition \ref{ext} of a local discrete extremum, we are in the position to state the DMP. 
\begin{lemma} (Discrete Maximum Principle) \label{discmax}
	Let $\mathcal{J} \subset \mathcal{N}$ be a given set, we let
	\begin{align}
	\mathcal{J}^*&:= \{j \in \mathcal{N}~:~\text{there exist }~i \in \mathcal{J} \text{ with a corner point } x_j \in \omega_i\}, \\ 
	\mathcal{J}_*&:=  \{j \in \mathcal{J}~:~1 \leq j \leq N\}.
	\end{align} Note that $\mathcal{J}_*$ is the collection of those corner points in $\mathcal{J}$ that are in interior of $\Omega$ and $\mathcal{J}^*$ denotes the collection of corner points which lie in the union of elements which support the vertices in $\mathcal{J}$. Let $K \in \cT_h$ and $v_h \in \cV_h$ be such that $v_h$ is locally minimal (resp., maximal) on an interior corner point $x_i$ in $K$ and if $\forall~i \in \mathcal{I}:= \{1,2, \cdots, N\},$ the following holds
	\begin{align}
	&\cA^{SIP}_K (v_h, \phi_{i}^K) \leq -\sum_{z_i \in e, e \in \partial K} \Gamma_e h_e^{-1} \int_e |\sjump{v_h}| - \delta_K h_K^{-1} \int_{K} |\nabla v^K_h|,
	\end{align} 
	where $\Gamma_e>0$ and $\delta_K>0$ are positive constants \\
	(resp., $\cA^{SIP}_K (v_h, \phi_{i}^K) \geq -\sum_{z_i \in e, e \in \partial K} \Gamma_e h_e^{-1} \int_e |\sjump{v_h}| - \delta_K h_K^{-1} \int_{K} |\nabla v^K_h| \quad i \in \mathcal{I}= \{1,2, \cdots, N\}$), then $v_h$ has no strict local discrete minimum (resp., maximum) in any interior corner point. Moreover, the global minimum (resp., maximum) is on the boundary, i.e.,
	\begin{align}
	v^K_h(x_j) \leq \max \{0, M\} \quad \forall~ j \in \mathcal{J}_*,
	\end{align}
	where $M:= \max\limits_{{j \in \mathcal{J}^* \setminus \mathcal{J}_*}} v_h^K(x_j)$.
\end{lemma}
\begin{remark} \label{ourdmp}
	In our analysis, we will be using the following particular form of the discrete maximum principle (cf.~Lemma \ref{discmax}). If
	\begin{align*}
	\cA^{SIP}_K (v_h, \phi_{i}^K) \leq~ 0~ \forall~i \in \mathcal{I}:= \{1,2, \cdots, N\},
	\end{align*}
	(resp., $\cA^{SIP}_K (v_h, \phi_{i}^K) \geq 0 \quad i \in \mathcal{I}= \{1,2, \cdots, N\}$), then $v_h$ has no strict local discrete minimum (resp., maximum) at any interior corner point.
\end{remark}
Next, for given $z \in L^2(\Omega)$, we discuss the improved $L^2$-error estimates for the obstacle problem \eqref{eq:C4}--\eqref{eq:C5}.

\begin{theorem} \label{L2er}
	Let the assumptions of Theorem \ref{asp1} and Proposition~\ref{prop:ctrl_reg} hold and additionally $g \in W^{2, \infty}(\Omega)$ holds. 
	Let $u$ and $u_h$ be the solutions of \eqref{eq:C4}--\eqref{eq:C5} and \eqref{disc1}--\eqref{disc2}, respectively. For any $\Omega_0 \subset \subset \Omega$, the following error estimate holds:
	\begin{align} \label{imL2}
	\|S(z)-S_h(z_h)\|_{L^2(\Omega_0)} &=\|u-u_h\|_{L^2(\Omega_0)} \notag \\ &\leq C h^{1+\beta} \big(\|u\|_{W^{2,p}(\Omega)} + \|g\|_{W^{2,\infty}(\Omega)} \big),
	\end{align}
	where $\beta:= 1-\frac{d}{p}$ and the positive constant $C$ is independent of $h$ and the penalty parameter $\eta >0$. Further, let $0< \epsilon< \frac{d}{2p}$, then \eqref{imL2} implies
	\begin{align} \label{impL2}
	\|u-u_h\|_{L^2(\Omega_0)} \leq C h^{2-2\epsilon} \big(\|u\|_{W^{2,p}(\Omega)}+ \|g\|_{W^{2,\infty}(\Omega)} \big).
	\end{align}
\end{theorem}
	\begin{proof}
	We note that
	\begin{align}
	\|u-u_h\|_{L^2(\Omega_0)} \leq \|u-Pu\|_{L^2(\Omega_0)} + \|Pu-u_h\|_{L^2(\Omega_0)}, \label{reqest}
	\end{align}
	where $P : H^2(\Omega) \cap H^1_0(\Omega) \rightarrow \cV_h$ is a projection operator which is defined as the unique solution to the following equation
	\begin{align} \label{proj}
	\cA^{SIP}(Pu,v_h) = - \int_{\Omega} \Delta u v_h~ dx \quad \forall~v_h \in \cV_h.
	\end{align}
	The operator $Pu \in \cV_h$ is indeed the SIPG finite element approximation to $u$ which solves the following Poisson equation
		\begin{align} \label{poi}
		a(u, v) = \int_{\Omega} w v ~dx \quad \forall~ v \in \cV,
		\end{align}
		where $w = -\Delta u$. 
	Using the standard duality arguments (see \cite[Section 10.5]{BScott:2008:FEM}) and the ideas in the article \cite[Theorem 5.1]{chen2004pointwise}, we have the following $L^2$-error estimate and maximum $(L^{\infty}$-norm) error estimate for the Poisson equation, i.e.,
	\begin{align}
	\|u-Pu\|_{L^{\infty}(\Omega)} &\leq Ch \inf_{v_h \in \cV_h} \| u-v_h\|_{W^{1,\infty}(\Omega)} \leq C h^{2-\frac{d}{p}} \|u\|_{W^{2,p}(\Omega)}\, , \ 1 \leq p < \infty, \label{Linf}\\ \|u-Pu\|_{L^{2}(\Omega)} &\leq Ch^2(\eta + \eta^{-5}), \label{L^2}
	\end{align}
	where $\eta$ is the penalty parameter defined in the $\cA^{SIP}(\cdot, \cdot)$ (equation \eqref{DP12}) and the positive constant $C$ (in \eqref{Linf} and \eqref{L^2}) is independent of $h$. Next, we employ the discrete maximum principle (Lemma \ref{discmax}) to deal with the second error term $\|Pu-u_h\|_{L^2(\Omega_0)} $ (in \eqref{reqest}). We will proceed in few steps\\
	
	\textbf{Step 1.}
	$\bullet$	Let $K \in \cT_h$ be a fixed element such that $K \cap \Omega_0 \neq \emptyset$ and note that both $Pu$ and $u_h$ belong to $\cV_h$. For $i \in \{1,2,3\}$, let $Pu^K_i$ and $u^K_{h,i}$ denote the components of the coefficient vectors $Pu^K$ and $u_h^K$, respectively. Next, we estimate the following differences
	\begin{subequations}
	\begin{align}
	Pu^K_i-u^K_{h,i} \quad \text{and} \quad u^K_{h,i}- Pu^K_i \quad \text{where } K \in \cT_h,
	\end{align}
	where $u^K_{h,i}= u_h^K(x_i)$ and $\{x_i\}_{i=1}^3$ are the vertices of $K$. Note that, we are interested in the maximum error estimates around the set $\Omega_0 \subset \subset \Omega$. Let us denote the continuous active set as $\cQ:= \{x \in \Omega~:~u(x)= g(x)\}$. Let $\{x_i\}^n_{i=1}$ and $\{x_i\}^m_{i=n+1}$ be the set of all interior and boundary corner points of the triangulization associated with $\Omega_0$, respectively. Next, we introduce the index sets
	\begin{align}
	&\cQ_1:= \{i \in \{1,2, \cdots, n\}~:~\omega_i \cap \cQ \neq \phi\},  \\ & \cL_1:= \{i \in \{1,2, \cdots, N\}~:~\omega_i \cap \cQ \neq \phi\},\\ &\cQ_2:= \{i \in \{1,2, \cdots, n\}~:~u^K_h(z_i) = g^K_h(z_i) \}, \\ &\cL_2:= \{i \in \{1,2, \cdots, N\}~:~u^{K}_h(z_i) = g^{K}_h(z_i) \},\\ &\cH_1:= \{1,2, \cdots, n\} \setminus \cQ_1 \quad \text{;} \quad \cH_2:= \{1,2, \cdots, n\} \setminus \cQ_2,  \\ &\mathcal{F}_1:= \{1,2, \cdots, N\} \setminus \cL_1 \quad \text{;} \quad \mathcal{F}_2:= \{1,2, \cdots, N\} \setminus \cL_2, \\ &\cQ_{h,1}:= \bigcup\limits_{i \in \cQ_1} \omega_i  \quad 
	\text{and} \quad \cQ_{h,2}:= \bigcup\limits_{i \in \cQ_2} \omega_i, \\  &\cL_{h,1}:= \bigcup\limits_{i \in \cL_1} \omega_i  \quad 
	\text{and} \quad \cL_{h,2}:= \bigcup\limits_{i \in \cL_2} \omega_i.
	\end{align}
	\end{subequations}
	
	\textbf{Step 2.}
	$\bullet$
	Next, we prove the following estimate for all $i \in \{1,2,\cdots,n,n+1,\cdots,m\},$ (i.e., the interior and boundary corner points of $\Omega_0$)
	\begin{align}
	Pu_i^K-u^K_{h,i} \leq \|Pu-u\|_{L^{\infty}(\cL_{h,1})} + \underset{i \in \cL_1}{\max} |u(x_i)- g_h(x_i)|.
	\end{align}
	Note that $\cQ_1 \cup \cH_1 \cup \{n+1, \cdots,m\}:=\{1,2,\cdots,n,n+1,\cdots,m\}$. In view of $\Omega_0 \subset \subset \Omega$, then observe that if $i \in \{n+1, \cdots,m\}$ then $i \in \{1,2,\cdots,N\}$ and the proof follows on the similar lines as in \cite[Lemma A.2 and A.3]{Meyer:2013:OCPOB}. Next, let $i \in \cQ_1$ which implies $u^K_h(x_i) \geq g^K_h(x_i)$ and $K \in \omega_i$ such that
	\begin{align*}
	Pu^K_i-u^K_{h,i} &= Pu^K_i-u(x_i)+u(x_i)-u^K_{h,i} \\ & \leq \|Pu-u\|_{L^{\infty}(\cQ_{h,1})} + u(x_i)-g^K_h(x_i) \\ & \leq \|Pu-u\|_{L^{\infty}(\cL_{h,1})} + \underset{i \in \cL_1}{\max} |u(x_i)- g_h(x_i)|.
	\end{align*}
	Now, let $i \in \cH_1$, i.e., $i \in \{1,2,\cdots,n\}$ and $i \notin \cQ_1$, then we have $u(x) - g(x) >0~~\forall~ x \in K \in \omega_i$ and using \eqref{eq:C5}, we have $\zeta_{K,i}(x)=0$. For $K \in \omega_i$ and with the help of \eqref{proj} and \eqref{eq:C4}, it holds that
	\begin{align*}
	0= \int_{K} (-\Delta u - z) \phi^K_{i} dx &= \int_{K} -\Delta u  \phi^K_{i} dx - \mathbf{\xA}^{SIP}_K(u_h,\phi^K_{i} ) + \zeta_{K,i} \\ &= \cA^{SIP}_K(Pu,\phi^K_{i}) - \mathbf{\xA}^{SIP}_K(u_h,\phi^K_{i} ) + \zeta_{K,i} \\ & \geq  \cA^{SIP}_K(Pu,\phi^K_{i}) - \mathbf{\xA}^{SIP}_K(u_h,\phi^K_{i} ) \\ & = \cA^{SIP}_K(Pu-u_h,\phi^K_{i})  \quad \forall i~ \in \cH_1.
	\end{align*}
	Using Remark \ref{ourdmp} and $\cH_1=(\cH_1)_*$, we deduce
	\begin{align}
	Pu^K_i-u^K_{h,i}  &\leq \max \bigg\{0, \underset{i \in \cH_1^*\setminus (\cH_1)_{*}}{\max} (Pu^K_i-u^K_{h,i} ) \bigg\} \\ & =\max \bigg\{0, \underset{i \in \cH_1^*\setminus \cH_1}{\max} (Pu^K_i-u^K_{h,i} ) \bigg\} \notag \\ & \leq  \underset{i \in \cL_1}{\max} (Pu^K_i-u^K_{h,i})  \hspace*{3cm}\qquad (\text{using } \cH_1^*\setminus \cH_1 \subseteq \cL_1) \\ & \leq \|Pu-u\|_{L^{\infty}(\cL_{h,1})} + \underset{i \in \cL_1}{\max} |u(x_i)- g_h(x_i)|. \label{eq}
	\end{align}
	Finally, we have the estimate for $Pu^K_i-u^K_{h,i}$.

	\textbf{Step 3.}
	$\bullet$ Using the similar arguments in Step 2 along with definitions of the sets $\cQ_2$ and $\cL_2$, we have the following estimate for $ u^K_{h,i}- Pu^K_i$
	\begin{align*}
	u^K_{h,i}- Pu^K_i  \leq \|u-Pu\|_{L^{\infty}(\cL_{h,2})} + \|g_h-g\|_{L^{\infty}(\cL_{h,2})}.
	\end{align*}
	
	\textbf{Step 4.} $\bullet$ Using Steps 1, 2 and 3, we conclude
	\begin{align}
	\|Pu-u_h\|_{L^{\infty}(\Omega_0)} \leq C \bigg( \|u-Pu\|_{L^{\infty}(\cL_{h,2} \cup \cL_{h,1})} + \|g_h-g\|_{L^{\infty}(\cL_{h,2})} \notag \\ + \underset{i \in \cL_1}{\max} |u(x_i)- g_h(x_i)| \bigg). \label{res1}
	\end{align}

	Lastly, we bound the third term on the right side of \eqref{res1}. Let $m \in \{0,1\}$ and $u, g \in C^{m, \beta}(\cL_{h,1})$. 
	Notice that this H\"older regularity follows from the respective regularity on $u$ and $g$. Since $j \in \cL_1$, there exists 
	a $y \in \omega_j \cap \cQ$ such 
	that $|x_j - x'| \leq Ch$ and $u(y)=g(y)$. Next, we show the following estimate for $m \in \{0,1\}$:
	\begin{align} \label{b3}
	0 \leq u(x_j)- g_h(x_j) \leq C h ^{m + \beta} \bigg( \|u\|_{C^{m, \beta}(\cL_{h,1})} + \|g\|_{C^{m, \beta}(\cL_{h,1})}  \bigg) +  \|g_h-g\|_{L^{\infty}(\cL_{h,1})},
	\end{align}
	where $C$ is a positive constant independent of $h$. For $m=0$, 
	\begin{align}
	u(x_j)- g_h(x_j) & = u(x_j)-u(y)+u(y)-g(x_j) + g(x_j)- g_h(x_j) \notag \\ & =  (u(x_j)-u(y))+(g(y)-g(x_j)) + (g(x_j)- g_h(x_j)) \notag\\ & \leq \big( \|u\|_{C^{0, \beta}(\cL_{h,1})} + \|g\|_{C^{0, \beta}(\cL_{h,1})}  \big)|z_j-y|^{\beta} + \|g_h-g\|_{L^{\infty}(\cL_{h,1})}. \label{b1}
	\end{align}
	Next, for $m=1$, we employ the same ideas as in articles \cite[Lemma A.5]{Meyer:2013:OCPOB} and \cite[Step 4, Pg. 660]{dond2016nonconforming} to conclude
	\begin{align}
	u(x_j)- g_h(x_j) \leq C h^{1 +\beta}\big( \|u\|_{C^{1, \beta}(\cL_{h,1})} + \|g\|_{C^{1, \beta}(\cL_{h,1})}  \big). \label{b2}
	\end{align}
	Using the following standard interpolation estimate \cite{BScott:2008:FEM}
	\begin{align*}
	\|g_h-g\|_{L^{\infty}(\cL_{h,1})} \leq C h^2\|g\|_{W^{2,\infty}(\Omega)}
	\end{align*}
	together with Steps 1, 2, 3 and 4 and inserting the Sobolov embedding $W^{2,p}(\Omega) \hookrightarrow C^{1, \beta}(\Omega)$ for $\beta= 1-\frac{d}{p}$ (or with equations \eqref{reqest}, \eqref{Linf}, \eqref{L^2}, \eqref{eq}, \eqref{res1} and \eqref{b3}), we have the proof of Theorem \ref{L2er}.
\end{proof}
\subsection{Error estimates for the state and control: global matching}
For the error estimates for the control and state in the global matching case 
($\widehat{\Omega} = \Omega$ in \eqref{problem1}), 
we first state the existence of a sequence of discrete optimal solutions to \eqref{disocp}, say $\{\tilde{z}_h\}_{h>0}$ such that $\{\tilde{z}_h\}_{h>0} \rightarrow \tilde{z}$ in $L^2(\Omega)$. We skip the proof for the sake of brevity, as it follows along the lines of \cite[Lemma 5.5]{Meyer:2013:OCPOB}.
\begin{lemma} \label{Seq}
	Let $\tilde{z} \in L^2(\Omega)$ satisfies \eqref{grpp1} (quadratic growth property) and Theorem \ref{main:thm1} holds. Then, we have existence of a sequence $\{\tilde{z}_h\}_{h>0}$ of the local optimal solutions to \eqref{disocp} which converge to $\tilde{z}$ in $L^2(\Omega)$ as $h \rightarrow 0$.
\end{lemma}
The sequence $\{\tilde{z}_h\}_{h>0}$ in Lemma \ref{Seq} is generated by standard arguments \cite{casas2002error,Meyer:2013:OCPOB} through solving the following auxiliary problem
\begin{equation} \label{problem3}
\begin{cases}
\min \quad F_h (z):= \frac{1}{2} \|S_h(z)-u_d\|^2_{L^2(\Omega)}+ \frac{\nu}{2} \|z\|^2_{L^2(\Omega)}, & \\  \text{ such that } \quad z \in B_{\gamma}(\tilde{z}),
\end{cases}
\end{equation}
where $B_{\gamma}(\tilde{z})$ is defined in Proposition \ref{grpp}. 
Next, we state and prove the control error estimate in the setting under consideration.
\begin{theorem} \label{main}
	Let $\tilde{z}$ satisfies Lemma \ref{Seq} and the error estimate \eqref{err1} holds. Then, there exists a constant $C>0$, independent of $h$, such that
		\begin{align}
		\|\tilde{z}-\tilde{z}_h\|_{L^2(\Omega)} & \leq C \sqrt{h(1+h)}, \label{contest}
		\end{align}
	where $h$ is sufficiently small and $\tilde{z}$ is the limit of a sequence $\{\tilde{z}_h\}_{h>0}$ of local solutions to \eqref{disocp}.
\end{theorem}
\begin{proof}
	In view of Lemma \ref{Seq}, we have existence of sequence of local solutions to \eqref{disocp} converging strongly to $\tilde{z}$ in $L^2(\Omega)$. 
		and hence, for sufficiently small choices of $h>0 $ and $\gamma$, $\tilde{z}_h$ satisfies
	\begin{align}
	F_h(\tilde{z}_h) \leq F_h(\tilde{z}) \quad\text{and} \quad \tilde{z}_h \in B_{\gamma}(\tilde{z}). \label{est}
	\end{align}
	The quadratic growth property (Proposition \ref{grpp}) and \eqref{est} yield the following
	\begin{align}
	F(\tilde{z}) & \leq F(\tilde{z}_h) - \kappa\|\tilde{z}_h-\tilde{z}\|^2_{L^2(\Omega)}  \implies \kappa \|\tilde{z}_h-\tilde{z}\|^2_{L^2(\Omega)}  \leq F(\tilde{z}_h) -F(\tilde{z}) \label{eq11} \\ \implies & \kappa \|\tilde{z}_h-\tilde{z}\|^2_{L^2(\Omega)} \leq F(\tilde{z}_h)- F_h(\tilde{z}_h) +  F_h(\tilde{z}_h) + F_h(\tilde{z})-F_h(\tilde{z}) -F(\tilde{z}) \notag \\  & \hspace*{1.8cm} \leq |F(\tilde{z}_h)- F_h(\tilde{z}_h) | + |F_h(\tilde{z}) -F(\tilde{z})|. \label{est1}
	\end{align}
	First, we bound the term $|F(\tilde{z}_h)- F_h(\tilde{z}_h) | $ using the definition \eqref{disocp} and exploit the error bound for term $\|S_h(\tilde{z}_h)-S(\tilde{z}_h)\|^2_{L^2(\Omega)}$ using \eqref{err1} and Remark \ref{rem1} (by taking right hand side as $\tilde{z}_h$ in the elliptic variational equality of Problem \eqref{problem1}).
	\begin{align}
	|F(\tilde{z}_h)- F_h(\tilde{z}_h) | & = \frac{1}{2} \bigg|\|S(\tilde{z}_h)-u_d\|^2_{L^2(\Omega)} - \|S_h(\tilde{z}_h)-u_d\|^2_{L^2(\Omega)} \bigg| \notag  \\ &  \hspace*{-0.8cm}= \frac{1}{2} \bigg|\|S(\tilde{z}_h)-u_d\|^2_{L^2(\Omega)} - \|S_h(\tilde{z}_h)-S(\tilde{z}_h) + S(\tilde{z}_h)-u_d\|^2_{L^2(\Omega)} \bigg| \notag \\ & \hspace*{-0.8cm} \leq \frac{1}{2} \|S_h(\tilde{z}_h)-S(\tilde{z}_h)\|^2_{L^2(\Omega)}  + \|S_h(\tilde{z}_h)-S(\tilde{z}_h)\|_{L^2(\Omega)}\|S(\tilde{z}_h)-u_d\|_{L^2(\Omega)}  \notag \\ & \hspace*{-0.8cm} \leq Ch^{2} + Ch\|S(\tilde{z}_h)-u_d\|_{L^2(\Omega)}. \label{laes}
	\end{align}
	Next, we prove the uniform boundedness of $\tilde{z}_h~\forall ~h>0$. Using \eqref{1.8c}, \eqref{1.8d}, \eqref{1.8e} and the coercivity of $\cA^{SIP}$, it holds that
	\begin{align} \label{k1}
	\|\tilde{z}_h\|_{L^2(\Omega)} & \leq \frac{1}{\nu} \|\tilde{p}_h\|_{L^2(\Omega)}
	\end{align}
	and
	\begin{align} 
	\|\tilde{p}_h\|^2_{L^2(\Omega)}  \leq \vertiii{\tilde{p}_h}^2 \leq
	\xA^{SIP}(\tilde{p}_h, \tilde{p}_h) &= \int_{\Omega} (\tilde{u}_h - u_d) \tilde{p}_h~dx + \langle \lambda_h , \tilde{p}_h \rangle_h \notag \\  & \leq \int_{\Omega} (\tilde{u}_h - u_d) \tilde{p}_h~dx  \notag \\ & \leq \|\tilde{u}_h - u_d\|_{L^2(\Omega)} \|\tilde{p}_h\|_{L^2(\Omega)}  \notag \\ & \leq \bigg(\|\tilde{u}_h\|_{L^2(\Omega)} +\|u_d\|_{L^2(\Omega)}\bigg)\|\tilde{p}_h\|_{L^2(\Omega)}. \label{k2}
	\end{align}
	The Lipschitz continuity of $S$ (see Lemma 2.1 in \cite{Kunisch:2012:SOPS}) and the strong convergence of $\tilde{z}_h$ (Lemma \ref{Seq}) provide the convergence of $\tilde{u}_h$ and hence $\tilde{u}_h$ is uniformly bounded. Using \eqref{k2}, the uniform boundedness of $\tilde{p}_h$ follows and from \eqref{k1} we have the uniform boundedness of $\tilde{z}_h$. Finally \eqref{laes} implies the following estimate
	\begin{align*}
	|F(\tilde{z}_h)- F_h(\tilde{z}_h) | \leq C h(1+h).
	\end{align*}
	We apply the same arguments for estimating the second term $ |F_h(\tilde{z}) -F(\tilde{z})|$ and this finishes the proof of Theorem \ref{main} using \eqref{est1}.
\end{proof}

For the state variable, we employ $S(\tilde{z})=\tilde{u}$ and $S_h(\tilde{z}_h)=\tilde{u}_h$, together with the triangle inequality, estimates \eqref{err1}, \eqref{contest} and the Lipschitz continuity of $S_h$ gives the following
\begin{align*}
\|\tilde{u}-\tilde{u}_h\|_h&=\|S(\tilde{z})-S_h(\tilde{z}_h)\|_h =\|S(\tilde{z})-S_h(\tilde{z})+S_h(\tilde{z})-S_h(\tilde{z}_h)\|_h \nonumber \\ 
& \leq \|S(\tilde{z})-S_h(\tilde{z})\|_h + \|S_h(\tilde{z})-S_h(\tilde{z}_h)\|_h  \leq Ch + \|\tilde{z}-\tilde{z}_h\|_{L^2(\Omega)} 
\nonumber \\ 
& \leq C \sqrt{h(1+h)}
\end{align*}
which immediately implies
\begin{align}
\|\tilde{u}-\tilde{u}_h\|_{L^2(\Omega)} & \leq C \sqrt{h(1+h)}.
\end{align}
\subsection{Error estimates for the state and control: local matching} \label{subsec}
In this subsection, we prove the quasi-optimal a priori error estimates of $\tilde{z}$ and $\tilde{u}$ 
which satisfies \eqref{problem1} for ${\widehat{\Omega}= \Omega_0 \subset \subset \Omega.}$ We present the main result of this subsection in the next theorem.
\begin{theorem} \label{main1} (Error estimate for control).
	Let $\epsilon$ be such that $0 < \epsilon < \frac{d}{2p}$. Then, the following estimate holds
	\begin{align}
	\|\tilde{z}-\tilde{z}_h\|_{L^2(\Omega)} & \leq C h^{1-\epsilon}, \label{contest1}
	\end{align} 
where $\tilde{z}$ is the control variable in \eqref{problem1}, $\{\tilde{z}_h\}_{h}$ is a sequence  which is obtained using Lemma \ref{Seq} (i.e., $\{\tilde{z}_h\}$ converges to $\tilde{z}$) and the positive constant $C$ is independent of $h$ and depends on the bounds $\|g\|_{W^{2, \infty}(\Omega)}$ and $\|u\|_{W^{2,p}(\Omega)}$.
\end{theorem}
\begin{proof}
	From estimate \eqref{est1}, we have
	\begin{align*}
	\kappa \|\tilde{z}_h-\tilde{z}\|^2_{L^2(\Omega)} \leq |F(\tilde{z}_h)- F_h(\tilde{z}_h) | + |F_h(\tilde{z}) -F(\tilde{z})|. 
	\end{align*}
	Using \eqref{laes}, it is sufficient to bound the term 
	$\|S_h(\tilde{z}_h)-S(\tilde{z}_h)\|_{L^2(\Omega_0)} $. Employing Theorem \ref{L2er} 
	and same ideas as in Theorem \ref{main}, it holds that 
	\begin{align*}
	\|S_h(\tilde{z}_h)-S(\tilde{z}_h)\|_{L^2(\Omega_0)} \leq Ch^{2-2\epsilon} \bigg( \|u\|_{W^{2,p}(\Omega)} + \|g\|_{W^{2, \infty}(\Omega)} \bigg),
	\end{align*}
	and we finally obtain the desired result 
	\begin{align}
	\|\tilde{z}-\tilde{z}_h\|_{L^2(\Omega)}  \leq C h^{1-\epsilon}.
	\end{align}
\end{proof}
We conclude this subsection by stating the error estimates for the state variable.
\begin{corollary} \label{main2} (Error estimate for state).
	Let $\tilde{u}$ and $\tilde{u}_h$ be the corresponding continuous and discrete state variables, respectively. Then, it holds that
	\begin{align}
	\|\tilde{u}-\tilde{u}_h\|_h \leq C h^{1-\epsilon}.
	\end{align}
\end{corollary}
\section{Conclusions} \label{sec6}
In this article, we have derived the a priori error estimates for the discontinuous Galerkin approximation to the optimal control of obstacle problem. The key ingredients of the analysis are the quadratic growth property, the discrete maximum property and improved a priori error estimates for solutions of the obstacle problem. The error estimates are quasi-optimal when $\widehat{\Omega} = \Omega_0 \subset\subset \Omega$ in \eqref{problem1}. As a part of future work, we plan to develop a posteriori error estimates using the results from this paper.
\appendix

\section{Proof} \label{sec7}

\begin{proof}
		 [\textbf{Proof of Lemma \ref{normeq}}]
	For any $w \in H^1(\Omega,\cT_h)$, define the function $v \in \cV \cap H^2(\Omega)$ that satisfies
	\begin{align*}
	-\Delta v&=w \quad \text{in } \Omega,\\ 
	v&=0 \quad \text{on } \partial\Omega.
	\end{align*}
	Using the integration by parts formula \cite[Appendix~C]{L.CEvans:2022} and discrete Cauchy-Schwartz inequality \cite[Appendix~C]{L.CEvans:2022}, it holds that
	{\small
		\begin{align*}
		\|w\|^2_{L^2(\Omega)} &= \int_{\Omega} w~(-\Delta v)~dx = \sum_{K \in \cT_h} \int_{K} \nabla w \cdot \nabla v ~dx - \sum_{e \in \cE_h} \int_e \sjump{w} \frac{\partial v}{\partial n}~ds
		\\ & \leq \sum_{K \in \cT_h}  \|\nabla w\|_{L^2(K)} \|\nabla v\|_{L^2(K)}+ \bigg(\sum_{e \in \cE_h} h_e^{-\frac{1}{2}} \big\|\sjump{w}\big\|_{L^2(e)} h_e^{\frac{1}{2}} \big\|\frac{\partial v}{\partial n}\big\|_{L^2(e)}  \bigg) 
		\\ & \leq \bigg(\sum_{K \in \cT_h} \|\nabla w\|^2_{L^2(K)} \bigg)^{\frac{1}{2}}  \bigg(\sum_{K \in \cT_h} \|\nabla v\|^2_{L^2(K)} \bigg)^{\frac{1}{2}} \\ & \hspace*{3cm}+ \bigg(\sum\limits_{ e \in \cE_h} \frac{1}{h_e}\big\|\sjump{w}\big\|^2_{L^2(e)} \bigg)^{\frac{1}{2}} \bigg(\sum\limits_{ e \in \cE_h} h_e\big\|\frac{\partial v}{\partial n}\big\|^2_{L^2(e)} \bigg)^{\frac{1}{2}} 
		\\ & \hspace*{-1cm} \lesssim  \bigg(\sum\limits_{ K \in \cT_h} \big\|\nabla w\big\|^2_{L^2(K)} + \sum\limits_{ e \in \cE_h} \frac{1}{h_e}\big\|\sjump{w}\big\|^2_{L^2(e)} \bigg)^{\frac{1}{2}} \bigg(\sum\limits_{ K \in \cT_h} \|\nabla v\|^2_{L^2(K)} + \sum\limits_{ e \in \cE_h} h_e\big\|\frac{\partial v}{\partial n}\big\|^2_{L^2(e)}\bigg)^{\frac{1}{2}}.
		\end{align*}}
	Employing the standard discrete trace inequality \cite[equation (10.3.8) on Pg. 282]{BScott:2008:FEM}, the following holds
	\begin{align*}
	h_e\big\|\frac{\partial v}{\partial n}\big\|^2_{L^2(e)} \leq C \|v\|^2_{W^{2,2}(K)},
	\end{align*} where $C$ is a positive constant. Finally using $-\Delta v=w$, we have the proof.
\end{proof}

\bibliographystyle{siamplain}
\bibliography{AKR}
\end{document}